\documentclass{article}%
\usepackage{amsmath}
\usepackage{amsfonts}
\usepackage{amssymb}
\usepackage{graphicx}
\usepackage[square,numbers,sort&compress]{natbib}
\usepackage[pagewise]{lineno}
\usepackage{color}%
\providecommand{\U}[1]{\protect \rule{.1in}{.1in}}
\pdfoutput=1

\newtheorem{theorem}{Theorem}[section]

\newtheorem{definition}[theorem]{Definition}

\newtheorem{example}[theorem]{Example}

\newtheorem{proposition}[theorem]{Proposition}
\newtheorem{remark}[theorem]{Remark}

\newenvironment{proof}[1][Proof]{\noindent \textbf{#1.} }{\  \rule{0.5em}{0.5em}}
\numberwithin{equation}{section}
\begin{document}

\title{Discrete-time approximation for stochastic optimal control problems under the
$G$-expectation framework}
\author{Lianzi Jiang\thanks{Zhongtai Securities Institute for Financial Studies,
Shandong University, Jinan, Shandong 250100, PR China. jianglianzi95@163.com.}}
\date{}
\maketitle

\textbf{Abstract}. In this paper, we propose a class of discrete-time
approximation schemes for stochastic optimal control problems under the
$G$-expectation framework. The proposed schemes are constructed recursively
based on piecewise constant policy. We prove the convergence of the discrete
schemes and determine the convergence rates. Several numerical examples are
presented to illustrate the effectiveness of the obtained results.


\textbf{MSC-classification}. 93-08 (49L20 49M25 60H10 65K10 93E20)

\section{Introduction}

Recently, Peng \cite{P2007,P2008,P2010} systematically introduced the notion
of $G$-expectation. Under the $G$-expectation framework, a new kind of
Brownian motion, called $G$-Brownian motion, was constructed. The related
stochastic calculus of It\^{o}'s type have been established. For a recent
account and development of $G$-expectation theory and its applications, we
refer the reader to \cite{HJ2021,HJY2014,RY2018,STZh2011,ZhC2012} and the
references therein.

It is well known that stochastic optimal control theory provides
a useful tool for mathematical financial models. However, in some utility models with ambiguity volatility,
it cannot be done within a probability space framework (see, e.g., \cite{EJ2013,EJ2014}). The $G$-expectation
framework does not require a probability space and is convenient to study
financial problems involving volatility uncertainty. Motivated by this, for
each fixed $T>0$, we consider a stochastic controlled system driven by
$G$-Brownian motion in the following form:
\[
\left \{
\begin{array}
[c]{ll}%
dX_{s}=b(s,X_{s},\alpha_{s})ds+\sigma(s,X_{s},\alpha_{s})dB_{s}+h(s,X_{s}%
,\alpha_{s})d\langle B\rangle_{s}, & s\in(t,T],\\
X_{t}=x. &
\end{array}
\right.
\]
The value function of the optimal control problem is defined by
\[
v(t,x)=\sup_{\alpha \in \mathcal{A}[t,T\mathcal{]}}\mathbb{\hat{E}}%
_{t,x}^{\alpha}\bigg[g(X_{T})+\int_{t}^{T}f(s,X_{s},\alpha_{s})ds\bigg],
\]
for $(t,x)\in \lbrack0,T]\times \mathbb{R}^{m}$. Under some standard
conditions, Hu et al. \cite{HJY2014} proved that the value function is the
viscosity solution of the following fully nonlinear Hamilton--Jacobi--Bellman
(HJB) equation:
\[
\left \{
\begin{array}
[c]{l}%
\partial_{t}v+\sup \limits_{a\in A}F\left(  t,x,D_{x}v,D_{x}^{2}v,a\right)
=0,\text{ \ }(t,x)\in \lbrack0,T)\times \mathbb{R}^{m},\\
v\left(  T,x\right)  =g\left(  x\right)  ,
\end{array}
\right.
\]
where
\begin{align*}
&  F\left(  t,x,D_{x}v,D_{x}^{2}v,a\right)  =G\left(  H\left(  t,x,D_{x}%
v,D_{x}^{2}v,a\right)  \right)  +\left \langle b\left(  t,x,a\right)
,D_{x}v\right \rangle +f\left(  t,x,a\right)  ,\\
&  H_{ij}\left(  t,x,D_{x}v,D_{x}^{2}v,a\right)  =\langle D_{x}^{2}v\sigma
_{i}\left(  t,x,a\right)  ,\sigma_{j}\left(  t,x,a\right)  \rangle+2\langle
D_{x}v,h_{ij}\left(  t,x,a\right)  \rangle.
\end{align*}

The basic framework for convergence of numerical schemes to viscosity
solutions of HJB equations was established by Barles and Souganidis
\cite{BS1991}. In particular, the provable order of convergence for the
second-order HJB equations was first obtained by Krylov in
\cite{Krylov1997,Krylov1999,Krylov2000}, and one of the main ideas is a method
named by himself \textquotedblleft shaking the coefficients\textquotedblright%
\ combined with a mollification procedure. This technique was further
developed by Barles and Jakobsen to apply to general monotone approximation
schemes (see \cite{BJ2002,BJ2005,BJ2007} and the references therein). Since
then, many results have been done on the numerical schemes for optimal control
problems and HJB equations (see, e.g.,
\cite{BF2014,CS2008,DJ2012,FTW2011,GZZ2015,HLZ2020,PR2020} and the references
therein). We also mention some important works for solving stochastic optimal
control problems by means of forward and backward stochastic differential
equations (see, e.g., \cite{ABYZ2020,ET2018,ETT2018,FZZ2020,GLTZhZh2017} and
the references therein). However, little attention has been paid to the
discrete schemes for stochastic optimal control problems under the
$G$-framework ($G$-SOCPs).

Owing to the sublinear nature of $G$-expectation, there is no density
representation for the $G$-normal distribution, which leads to the failure of
the classical Markov chain approximation (see, e.g., \cite{KD2001,Krylov1999}%
). Note that Peng \cite{P2010} established the central limit theorem in the
$G$-expectation framework ($G$-CLT), which provides a theoretical foundation
for approximating $G$-normal distributed random variables. In particular, an
outstanding contribution to the convergence rate of $G$-CLT was by Krylov
\cite{Krylov2020} using stochastic control method under different model
assumptions. Let us also note other important works on this topic. In
\cite{HL2020}, Huang and Liang studied the convergence rate for a more general
$G$-CLT via a monotone approximation scheme and obtained an explicit bound of
Berry-Esseen type. Fang et al. \cite{FPSS2019} and Song \cite{Song2020} proved
the convergence rate of $G$-CLT by using Stein's method.

In this paper, we devote ourselves to designing a discrete-time approximation
scheme for $G$-SOCPs. With the help of piecewise constant policy, Euler
time-stepping, and $G$-CLT, we propose a general discrete-time approximation
scheme for solving $G$-SOCPs. By choosing the parameter set, two different
kinds of numerical schemes are derived. We also establish the convergence
rates for the numerical schemes by using the \textquotedblleft shaking the
coefficients\textquotedblright \ method. The result shows that our discrete
scheme has a 1/6 order rate in the general case and a better rate of order 1/4
in the special case. We remark that a similar convergence result was acquired
in \cite{HL2020}, but with different problem settings and discrete schemes.
Several numerical examples are demonstrated to illustrate the obtained
theoretical findings.

The rest of the paper is organized as follows. In Section 2, we review some
basic notations and results on the $G$-expectation theory. We propose the
discrete-time approximation scheme for $G$-SOCPs in Section 3. The convergence
rate of the proposed scheme is proved in Section 4. Numerical examples are
given in Section 5.

\section{Preliminaries}

This section briefly introduces some notions and preliminaries in the
$G$-framework. For more details, we refer the reader to
\cite{P2007,P2008,P2010} and the references therein.

\begin{definition}
Let $\Omega$ be a given set and let $\mathcal{H}$ be a linear space of real
valued functions defined on $\Omega$, satisfies $c\in \mathcal{H}$ for each
constant $c$ and $\left \vert X\right \vert \in \mathcal{H}$ if $X\in \mathcal{H}%
$. $\mathcal{H}$ is considered as the space of random variables. A functional
$\mathbb{\hat{E}}$: $\mathcal{H}\rightarrow \mathbb{R}$ is called a sublinear
expectation: if for all $X,Y$ $\in \mathcal{H}$, it satisfies the following properties:

\begin{description}
\item[(i)] Monotonicity: If $X\geq Y$ then $\mathbb{\hat{E}}\left[  X\right]
\geq \mathbb{\hat{E}}\left[  Y\right]  ;$

\item[(ii)] Constant preservation: $\mathbb{\hat{E}}\left[  c\right]  =c;$

\item[(iii)] Sub-additivity: $\mathbb{\hat{E}}\left[  X+Y\right]
\leq \mathbb{\hat{E}}\left[  X\right]  +\mathbb{\hat{E}}\left[  Y\right]  ;$

\item[(iv)] Positive homogeneity: $\mathbb{\hat{E}}\left[  \lambda X\right]
=\lambda \mathbb{\hat{E}}\left[  X\right]  $ for each $\lambda>0.$
\end{description}
\end{definition}

The triple $(\Omega,\mathcal{H},\mathbb{\hat{E})}$ is called a sublinear
expectation space.

\begin{definition}
Let $X_{1}$ and $X_{2}$ be two $n$-dimensional random vectors defined
respectively in sublinear expectation spaces $(\Omega_{1},\mathcal{H}%
_{1},\mathbb{\hat{E}}_{1})$ and $(\Omega_{2},\mathcal{H}_{2},\mathbb{\hat{E}%
}_{2})$. They are called identically distributed, denoted by $X_{1}\overset
{d}{=}X_{2}$, if $\mathbb{\hat{E}}_{1}\left[  \varphi(X_{1})\right]
=\mathbb{\hat{E}}_{2}\left[  \varphi(X_{2})\right]  $, for all $\varphi \in
C_{b,Lip}(\mathbb{R}^{n})$, the space of bounded Lipschitz continuous
functions on $\mathbb{R}^{n}$.
\end{definition}

\begin{definition}
In a sublinear expectation space $(\Omega,\mathcal{H},\mathbb{\hat{E})}$, a
random vector $Y=(Y_{1,\ldots,}Y_{n})$, $Y_{i}\in \mathcal{H}$, is said to be
independent from another random vector $X=(X_{1,\ldots,}X_{m})$, $X_{i}%
\in \mathcal{H}$ under $\mathbb{\hat{E}}\left[  \cdot \right]  $, denoted by
$Y\perp X$, if for every test function $\varphi \in C_{b,Lip}(\mathbb{R}%
^{m}\times \mathbb{R}^{n})$ we have $\mathbb{\hat{E}}\left[  \varphi
(X,Y)\right]  =\mathbb{\hat{E}[\hat{E}[\varphi(}x,Y\mathbb{)]}_{x=X}]$.
\end{definition}

\begin{definition}
\label{G normal def}A $d$-dimensional random vector $X=(X_{1},\ldots,X_{d})$
in a sublinear expectation space $(\Omega,\mathcal{H},\mathbb{\hat{E})}$ is
called $G$-normal distributed if for each $a,b\geq0$ we have%
\[
aX+b\bar{X}=\sqrt{a^{2}+b^{2}}X,
\]
where $\bar{X}$ is an independent copy of $X$, $G:\mathbb{S(}d\mathbb{)}%
\rightarrow \mathbb{R}$ denotes the function
\[
G\left(  A\right)  :=\frac{1}{2}\mathbb{\hat{E}[}\left \langle
AX,X\right \rangle ],\text{ \ }A\in \mathbb{S(}d\mathbb{)},
\]
where $\mathbb{S(}d\mathbb{)}$ denotes the collection of $d\times d$ symmetric matrices.
\end{definition}

\begin{definition}
Let $\Omega_{T}=C_{0}([0,T];\mathbb{R}^{d})$ be the space of $\mathbb{R}^{d}%
$-valued continuous paths on $[0,T]$ with $\omega_{0}=0$, endowed with the
supremum norm, and $B_{t}(\omega)=\omega_{t}$ be the canonical process. Set
\[
Lip(\Omega_{T})=\{ \varphi(B_{t_{1}},\ldots,B_{t_{n}}):n\geq1,t_{1}%
,\ldots,t_{n}\in \lbrack0,T],\varphi \in C_{b,Lip}(\mathbb{R}^{d\times n})\}
\text{,}%
\]
where $C_{b,Lip}(\mathbb{R}^{d\times n})$ denotes the set of bounded Lipschitz
functions on $\mathbb{R}^{d\times n}$. $G$-expectation on $(\Omega
_{T},Lip(\Omega_{T}))$ is a sublinear expectation defined by%
\[
\mathbb{\hat{E}[}X\mathbb{]=\tilde{E}[}\varphi(\sqrt{t_{1}-t_{0}}\xi
_{1},\ldots,\sqrt{t_{n}-t_{n-1}}\xi_{n})\mathbb{]},
\]
for all $X=\varphi(B_{t_{1}}-B_{t_{0}},\ldots,B_{t_{n}}-B_{t_{n-1}})$, where
$\xi_{1},\ldots,\xi_{n}$ are identically distributed $d$-dimensional
$G$-normal distributed random vectors in a sublinear expectation space
$(\tilde{\Omega},\mathcal{\tilde{H}},\mathbb{\tilde{E})}$ such that $\xi
_{i+1}$ is independent from $(\xi_{1},\ldots,\xi_{i})$, $i=1,\ldots,n-1.$ The
corresponding canonical process $B_{t}=(B_{t}^{i})_{i=1}^{d}$ is called a
$G$-Brownian motion and $(\Omega_{T},Lip(\Omega_{T}),\mathbb{\hat{E})}$ is
called a $G$-expectation space.
\end{definition}

\begin{definition}
Assume that $X\in Lip(\Omega_{T})$ has the representation $X=$ $\varphi
(B_{t_{1}}-B_{t_{0}},\ldots,B_{t_{n}}-B_{t_{n-1}})$. The conditional
$G$-expectation $\mathbb{\hat{E}}_{t_{i}}$ of $X$ is defined by, for some
$1\leq i\leq n$,
\[
\mathbb{\hat{E}}_{t_{i}}[\varphi(B_{t_{1}}-B_{t_{0}},\ldots,B_{t_{n}%
}-B_{t_{n-1}})]=\tilde{\varphi}(B_{t_{1}}-B_{t_{0}},\ldots,B_{t_{i}%
}-B_{t_{i-1}}),
\]
where%
\[
\tilde{\varphi}(x_{1},\ldots,x_{i})=\mathbb{\hat{E}}_{t_{i}}[\varphi
(x_{1},\ldots,x_{i},B_{t_{i+1}}-B_{t_{i}},\ldots,B_{t_{n}}-B_{t_{n-1}})].
\]

\end{definition}

For each given $p\geq1$, define $\Vert X\Vert_{L_{G}^{p}}=(\mathbb{\hat{E}%
}[|X|^{p}])^{1/p}$ for $X\in Lip(\Omega_{T})$, and denote by $L_{G}^{p}%
(\Omega_{T})$ the completion of $Lip(\Omega_{T})$ under the norm $\Vert
\cdot \Vert_{L_{G}^{p}}$. Then for $t\in \left[  0,T\right]  $, $\mathbb{\hat
{E}}_{t}[\cdot]$ can be extended continuously to the completion $L_{G}%
^{1}(\Omega_{T})$ of $Lip(\Omega_{T})$ under the norm $\left \Vert
\cdot \right \Vert _{L_{G}^{1}}$.

\begin{theorem}
[\cite{DHP2011,HP2009}]\label{E repesent}There exists a weakly compact family
$\mathcal{P}$ of probability measures on $\left(  \Omega_{T},\mathcal{B}%
(\Omega_{T})\right)  $ such that
\[
\mathbb{\hat{E}}[X]=\sup_{P\in \mathcal{P}}\mathbb{E}_{P}[X],\text{ for all
}X\in L_{G}^{1}(\Omega_{T}).
\]
$\mathcal{P}$ is called a set that represents $\mathbb{\hat{E}}$.
\end{theorem}

In the sequel, let $(B_{t})_{t\geq0}$ be a one dimensional $G$-Brownian motion
with $G(a)=\frac{1}{2}(\bar{\sigma}^{2}a^{+}-\underline{\sigma}^{2}a^{-})$,
where $\mathbb{\hat{E}}[B_{1}^{2}]=\bar{\sigma}^{2}$, $-\mathbb{\hat{E}} [-B_{1}
^{2}]=\underline{\sigma}^{2}$, $0\leq \underline{\sigma}\leq \bar{\sigma}
<\infty$. Denis et al. \cite{DHP2011} gave the following concrete set
$\mathcal{P}_{M}$ that represents $\mathbb{\hat{E}}$:%
\[%
\begin{array}
[c]{r}%
\mathcal{P}_{M}:=\left \{  P^{\theta}:P^{\theta}=P\circ \big(\int_{0}^{t}%
\theta_{s}dW_{s}\big)^{-1},\; \; \theta_{s}\in \Sigma \right \}  ,
\end{array}
\]
where $W_{\cdot}$ is a one dimensional Brownian motion on a probability space
$(\Omega,\mathcal{F},P)$ with filtration $\mathbb{F=\{}\mathcal{F}%
_{t}\mathcal{\}}_{t\geq0}$,\ and $\Sigma$ is the collection of all
$[\underline{\sigma},\overline{\sigma}]$-valued $\mathbb{F}$-adapted processes
$(\theta_{s})_{0\leq s\leq T}$.

Let $\{t_{0},t_{1},\ldots,t_{N}\}$ be a sequence of partitions of $[0,T]$ and
set $M_{G}^{0}(0,T)=\{ \eta_{t}(\omega)=\sum_{n=0}^{N-1}\xi_{n}(\omega
)I_{[t_{n},t_{n+1})}$ $(t):\xi_{n}\in Lip(\Omega_{t_{n}})\}$. For each given
$p\geq1$, denote by $M_{G}^{p}(0,T)$ the completion of $M_{G}^{0}(0,T)$ with
the norm $\Vert \eta \Vert_{M_{G}^{p}}=(\mathbb{\hat{E}}[\int_{0}^{T}|\eta
_{s}|^{p}ds])^{1/p}$. Then, for any $\xi_{t}\in$ $M_{G}^{2}(0,T)$ and
$\eta_{t}\in$ $M_{G}^{1}(0,T)$, the $G$-It\^{o} integral $\int_{0}^{T}\xi
_{t}dB_{t}$ and $\int_{0}^{T}\eta_{t}d\langle B\rangle_{t}$ are well defined,
see \cite{P2007,P2008,P2010} for more details.

\begin{proposition}
\label{Prop Ito's integral}For each $\xi_{t}\in$ $M_{G}^{2}(0,T)$ and
$\eta_{t}\in$ $M_{G}^{1}(0,T)$, we have

\begin{description}
\item[(i)] $\mathbb{\hat{E}}\left[  \int_{0}^{T}\xi_{t}dB_{t}\right]  =0;$
$\  \mathbb{\hat{E}}\left[  \big(\int_{0}^{T}\xi_{t}dB_{t}\big)^{2}\right]
=\mathbb{\hat{E}}\left[  \int_{0}^{T}\xi_{t}^{2}d\langle B\rangle_{t}\right]
;$

\item[(ii)] $\mathbb{\hat{E}}\left[  \left \vert \int_{0}^{T}\eta_{t}d\langle
B\rangle_{t}\right \vert \right]  \leq \overline{\sigma}^{2}\mathbb{\hat{E}%
}\left[  \int_{0}^{T}\left \vert \eta_{t}\right \vert dt\right]  .$
\end{description}
\end{proposition}

We end this section with Krylov's regularization results. Take a nonnegative
$\zeta \in C_{0}^{\infty}(\mathbb{R}^{m+1})$ with unit integral and support in
$\{(t,x):-1<t<0,|x|<1\}$ and for $\varepsilon \in(0,1)$ let $\zeta
_{\varepsilon}(t,x)=\varepsilon^{-m-2}\zeta(t/\varepsilon^{2},x/\varepsilon)$.
For locally integrable $u(t,x)$, we denote the mollification of $u$ by
\[
u^{(\varepsilon)}(t,x)=u(t,x)\ast \zeta_{\varepsilon}(t,x).
\]
Furthermore, if $u$ is 1/2-H\"{o}lder continuity in time and Lipschitz
continuity in space, then the standard properties of mollifiers (see Lemma 3.5
in \cite{Krylov1999} or C.4 in \cite{Evans2010}) indicate that
\begin{equation}
\left \Vert u-u^{(\varepsilon)}\right \Vert _{\infty}\leq C\varepsilon \text{
\ and \ }\left \Vert \partial_{t}^{l}D_{x}^{k}u^{(\varepsilon)}\right \Vert
_{\infty}\leq C\varepsilon^{1-2l-k}\text{ \ for }k+l\geq1. \label{bound}%
\end{equation}

\section{Discrete-time approximation}

In this section, we present a class of discrete-time approximation schemes for
solving G-SOCPs.

\subsection{Formulation of the problem}

We first give the definition of admissible controls.

\begin{definition}
For each $t\in \lbrack0,T]$, $\alpha$ is said to be an admissible control on
$[t,T]$, if it satisfies the following conditions:

\begin{description}
\item[(i)] $\alpha:[t,T]\times \Omega \rightarrow A$, where $A$ is a compact set
of $\mathbb{R}^{q};$

\item[(ii)] $\alpha \in M_{G}^{2}(t,T;\mathbb{R}^{q}).$
\end{description}
\end{definition}

The set of admissible controls on $[t,T]$ is denoted by $\mathcal{A}%
[t,T\mathcal{]}$.

Let $t\in \lbrack0,T]$, $x\in \mathbb{R}^{m}$, and $\alpha \in \mathcal{A}%
[t,T\mathcal{]}$. Consider the following $m$-dimensional controlled system:%
\begin{equation}
\left \{
\begin{array}
[c]{ll}%
dX_{s}^{\alpha,t,x}=b(s,X_{s}^{\alpha,t,x},\alpha_{s})ds+\sigma(s,X_{s}%
^{\alpha,t,x},\alpha_{s})dB_{s}+h(s,X_{s}^{\alpha,t,x},\alpha_{s})d\langle
B\rangle_{s}, & s\in(t,T],\\
X_{s}^{\alpha,t,x}=x, &
\end{array}
\right.  \label{1}%
\end{equation}
and%
\begin{equation}
v(t,x)=\sup_{\alpha \in \mathcal{A}[t,T\mathcal{]}}\mathbb{\hat{E}}%
_{t,x}^{\alpha}\bigg[g(X_{T})+\int_{t}^{T}f(s,X_{s},\alpha_{s})ds\bigg],
\label{2}%
\end{equation}
where $B_{\cdot}$ is a one dimensional $G$-Brownian motion defined in the
$G$-expectation space $(\Omega_{T},L_{G}^{1}(\Omega_{T}),\mathbb{\hat{E}})$,
$\langle B\rangle_{\cdot}$ is the quadratic variation process of the
$G$-Brownian motion, $b,\sigma,h$ are deterministic functions with values in
$\mathbb{R}^{m}$, and $g,f$ are two $\mathbb{R}$-valued functions. The
indices\ $\alpha,t,x$ on the $G$-expectation $\mathbb{\hat{E}}$ indicate that
the state process $X_{\cdot}$ depends on the starting point $(t,x)$ and the
control $\alpha$.

We need the following standard assumptions on the optimal control problem:

\begin{description}
\item[(H1)] The maps $b,\sigma,h:[0,\infty)\times \mathbb{R}^{m}\times
A\rightarrow \mathbb{R}^{m}$ are continuous, and there exists a constant $L>0$
such that, for any $t,s\in \lbrack0,\infty)$ $x,y\in \mathbb{R}^{m}$, and $a\in
A$,
\[
|{\varphi}(t,x,a)-{\varphi}(s,y,a)|\leq L\left(  |t-s|^{\frac{1}{2}%
}+|x-y|\right)  \  \text{and}\ |{\varphi}(t,x,a)|\leq L\text{,}\  \  \text{for
}{\varphi}\text{$={b},\sigma$},{h}\text{.}%
\]

\item[(H2)] The maps $f:[0,\infty)\times \mathbb{R}^{m}\times A\rightarrow
\mathbb{R}$ and $g:\mathbb{R}^{m}\rightarrow \mathbb{R}$ are continuous, and
there exists a constant $L>0$ such that, for any $t,s\in \lbrack0,\infty)$,
$x,y\in \mathbb{R}^{m}$, and $a\in A$,
\[
|{g}(x)-g(y)|+|{f}(t,x,a)-{f}(s,y,a)|\leq L\left(  |t-s|^{\frac{1}{2}%
}+|x-y|\right)  \text{ and }|{f}(t,x,a)|\leq L.
\]

\end{description}

Let $N\in \mathbb{N}$ and $\Delta=T/N>0$. We introduce a time mesh
$t_{n}=n\Delta$ for $n=0,1,\ldots,N$. Let $\mathcal{A}_{h}[t,T]$ be the subset
of $\mathcal{A}[0,T]$ consisting of all processes $\alpha_{\cdot}$ which are
constant in the intervals $[t_{n},t_{n+1})$, $n=0,1,\ldots,N-1$. For
simplicity, we will identify any $\alpha \in \mathcal{A}_{h}$ by the sequence of
random variables $a_{i}$ taking values in $A$ and denote $\alpha=(a_{0,}%
a_{1},\ldots,a_{N-1})$. For any $(t,x)\in \lbrack0,T]\times \mathbb{R}^{m}$,
define the value function corresponding to the piecewise constant control set
$\mathcal{A}_{h}[t,T]$%
\[
v_{\Delta}(t,x)=\sup_{\alpha \in \mathcal{A}_{h}[t,T]}\mathbb{\hat{E}}%
_{t,x}^{\alpha}\bigg[g(X_{T})+\int_{t}^{T}f(s,X_{s},\alpha_{s})ds\bigg].
\]
Since $\mathcal{A}_{h}[t,T]\subset \mathcal{A}[t,T]$, it is clear that
$v(t,x)-v_{\Delta}(t,x)\geq0$.

\begin{theorem}
\label{lemma1}Assume that \emph{(H1)-(H2)} hold. Then there exists a constant
$C$ independent of $\Delta$ such that%
\[
v(t,x)-v_{\Delta}(t,x)\leq C\Delta^{\frac{1}{4}},
\]
for any $\left(  t,x\right)  \in \lbrack0,T]\times \mathbb{R}^{n}$.
\end{theorem}

This theorem is proved in Appendix, and it is based on the \textquotedblleft
shaking the coefficients\textquotedblright \ method of Krylov in
\cite{Krylov1999} and the improved result of Jakobsen et al. in \cite{JPR2019}.

\subsection{Euler scheme}

We now give an approximation of the process $X_{\cdot}$ by the Euler method.
In what follows, we shall denote $\mathcal{A}_{h}[t,T]$ as $\mathcal{A}_{h}$
without causing confusion. For any fixed $\alpha=(a_{0,}a_{1},\ldots
,a_{N-1})\in \mathcal{A}_{h}$, we denote $\Delta B:=B_{t_{n+1}}-B_{t_{n}}\ $and
$\Delta \langle B\rangle:=\langle B\rangle_{t_{n+1}}-\langle B\rangle_{t_{n}}$
and define the Euler approximation process $\bar{X}_{\cdot}^{\alpha,t_{n},x}$
recursively by%
\begin{equation}
\bar{X}_{t_{i+1}}=\bar{X}_{t_{i}}+b(t_{i},\bar{X}_{t_{i}},a_{i})\Delta
+\sigma(t_{i},\bar{X}_{t_{i}},a_{i})\Delta B+h(t_{i},\bar{X}_{t_{i}}%
,a_{i})\Delta \langle B\rangle, \label{Euler_SDE}%
\end{equation}
for $i=n,\ldots,N-1$, with $\bar{X}_{t_{n}}=x$. Under the assumptions
(H1)-(H2), similar to the proof in \cite{KP1992}, one can obtain that%
\begin{equation}
\mathbb{\hat{E}}\bigg[\sup_{t\in \lbrack t_{n},t_{n+1}]}\left \vert
X_{t}^{\alpha,t_{n},x}-\bar{X}_{t_{n}}^{\alpha,t_{n},x}\right \vert
^{2}\bigg]\leq C\Delta,\; \label{Euler_1/2 estimate}%
\end{equation}
for any $\alpha \in \mathcal{A}_{h}$, $n=0,1,\ldots,N-1$, where $C$ is a
positive constant independent of $\Delta$.

For any $n=0,1,\ldots,N-1$, $x\in \mathbb{R}^{m}$, define
\begin{equation}
\bar{v}_{\Delta}(t_{n},x)=\sup_{\alpha \in \mathcal{A}_{h}}\mathbb{\hat{E}%
}_{t_{n},x}^{\alpha}\bigg[g(\bar{X}_{T})+\sum_{i=n}^{N-1}f(t_{i},\bar
{X}_{t_{i}},a_{i})\Delta \bigg]. \label{v^bar_h}%
\end{equation}
Furthermore, $\bar{v}_{\Delta}$ satisfies the following dynamic programming
principle (DPP):
\begin{equation}%
\begin{split}
&  \bar{v}_{\Delta}(T,x)=g(x),\\
&  \bar{v}_{\Delta}(t_{n},x)=\sup_{a\in A}\mathbb{\hat{E}}_{t_{n},x}%
^{a}\left[  \bar{v}_{\Delta}(t_{n+1},\bar{X}_{t_{n+1}})+f(t_{n},x,a)\Delta
\right]  ,
\end{split}
\label{DP v^bar_h}%
\end{equation}
for $n=0,\ldots,N-1$, $x\in \mathbb{R}^{m}$. With the help of Theorem
\ref{lemma1} and (\ref{Euler_1/2 estimate}), it is easy to check that for any
$n=0,\ldots,N-1$, $x\in \mathbb{R}^{m}$,
\begin{equation}
|v(t_{n},x)-\bar{v}_{\Delta}(t_{n},x)|\leq|v(t_{n},x)-v_{\Delta}%
(t_{n},x)|+|v_{\Delta}(t_{n},x)-\bar{v}_{\Delta}(t_{n},x)|\leq C\Delta
^{\frac{1}{4}}. \label{Euler approxi}%
\end{equation}

Define $B_{1}:=\{x\in \mathbb{R}^{m}:|x|<1\}$\ and $B=A\times \{(\mu
,\lambda):\mu \in(-1,0),\lambda \in B_{1}\}$. Extend $b,\sigma,f$ for negative
$t$ following the example $b(t,x,\alpha)=b(0,x,\alpha)$. For a fixed
$\varepsilon \in(0,1]$ and any $\beta=(\alpha,\mu,\lambda)\in B$ let
\[
b(t,x,\beta)=b_{\varepsilon}(t,x,\beta)=b(t+\varepsilon^{2}\mu,x+\varepsilon
\lambda,\alpha),
\]
and similarly define $\sigma(t,x,\beta)$ and $f(t,x,\beta)$. We denote by
$\mathcal{B}_{h}$ the set of $B$-valued progressively measurable processes
which are constant in each time interval. Now we proceed with the
regularization of the value function $\bar{v}_{\Delta}$. Letting
$S=T+\varepsilon^{2}$ and $N_{s}=[S/\Delta]$, for any $x\in \mathbb{R}^{m}$,
$n=0,\ldots,N_{s}$, define the following \textquotedblleft shaken" value
function:%
\[
\bar{u}_{\Delta}(t_{n},x)=\sup_{\beta \in \mathcal{B}_{h}}\mathbb{\hat{E}%
}_{t_{n},x}^{\beta}\bigg[g(\bar{X}_{t_{N_{s}}})+\sum_{i=n}^{N_{s}}f(t_{i}%
,\bar{X}_{t_{i}},\beta_{i})\Delta \bigg],
\]
where $\bar{X}_{\cdot}^{\beta,t_{n},x}$ is recursively defined by
\[
\bar{X}_{t_{i+1}}=\bar{X}_{t_{i}}+b(t_{i},\bar{X}_{t_{i}},\beta_{i}%
)\Delta+\sigma(t_{i},\bar{X}_{t_{i}},\beta_{i})\Delta B+h(t_{i},\bar{X}%
_{t_{i}},\beta_{i})\Delta \langle B\rangle,
\]
for $i=n,\ldots,N_{s}-1$ with $\bar{X}_{t_{n}}=x$. Moreover, the shaken value
function satisfies the following DPP:%
\begin{equation}
\bar{u}_{\Delta}(t_{n},x)=\sup_{\beta \in B}\mathbb{\hat{E}}_{t_{n},x}^{\beta
}\bigg[\bar{u}_{\Delta}(t_{n+1},\bar{X}_{t_{i+1}})+f(t_{n},x,\beta
)\Delta \bigg], \label{u^bar DPP}%
\end{equation}
for $n=0,1,\ldots,N_{s}-1$, $x\in \mathbb{R}^{m}$.

It is convenient to extend $\bar{v}_{\Delta}$ and $\bar{u}_{\Delta}\ $defined
on the grid points to the whole of $[0,T]$ and $[0,S]$, respectively, keeping
its values on $\{t_{n}\}_{0\leq n\leq N_{s}}$ and making it equal to the value
at $t_{n}$ on each interval $[t_{n},t_{n+1})$. The following standard result
is obtained by referring to the proof of Lemma 2.2 of \cite{Krylov2020} and
using the estimate of Proposition 5.3.1 of \cite{P2010}.

\begin{proposition}
\label{prop u^-}Assume that \emph{(H1)-(H2)} hold. Then, there exists a
constant $C\geq0$ such that for any $t\in \lbrack0,T]$, $x\in \mathbb{R}^{m}$,
\[
|\bar{u}_{\Delta}(t,x)-\bar{v}_{\Delta}(t,x)|\leq C\varepsilon,
\]
and for any $t,s\in \lbrack0,S]$, $x,y\in \mathbb{R}^{m}$
\[
|\bar{u}_{\Delta}(t,x)-\bar{u}_{\Delta}(s,y)|\leq C(|t-s|^{\frac{1}{2}}%
+\Delta^{\frac{1}{2}}+|x-y|).
\]

\end{proposition}

Then we give some results about the mollified function $\bar{u}_{\Delta
}^{(\varepsilon)}$.

\begin{proposition}
\label{prop u^epsi-}Assume that \emph{(H1)-(H2)} hold. Then,\

\begin{description}
\item[(i)] there exists a constant $C\geq0$ such that for any $t\in
\lbrack0,S]$, $x\in \mathbb{R}^{m},$
\[
|\bar{u}_{\Delta}(t,x)-\bar{u}_{\Delta}^{(\varepsilon)}(t,x)|\leq
C\varepsilon \text{;}%
\]

\item[(ii)] the function $\bar{u}_{\Delta}^{(\varepsilon)}\in C_{0} ^{\infty
}([0,S]\times \mathbb{R}^{m})$ and
\[
\left \Vert \partial_{t}^{l}D_{x}^{k}\bar{u}_{\Delta}^{(\varepsilon
)}\right \Vert _{\infty}\leq C\varepsilon^{1-2l-k}\text{, \ for }k+l\geq1;
\]

\item[(iii)] $\bar{u}_{\Delta}^{(\varepsilon)}$ satisfies the following
super-dynamic programming principle: for any $n=0,\ldots,N-1$, $x\in
\mathbb{R}^{m}$,
\[
\bar{u}_{\Delta}^{(\varepsilon)}(t_{n},x)\geq \sup_{a\in A}\mathbb{\hat{E}%
}_{t_{n},x}^{a}\left[  \bar{u}_{\Delta}^{(\varepsilon)}(t_{n+1},\bar
{X}_{t_{n+1}})+f(t_{n},x,a)\Delta \right]  .
\]

\end{description}

\begin{proof}
Properties (i)-(ii) are immediate from the properties of mollifiers
(\ref{bound}). Next, we only prove (iii). In view of (\ref{u^bar DPP}) and
noting that $\bar{X}_{t_{n+1}-\varepsilon^{2}\mu}^{\beta,t_{n}-\varepsilon
^{2}\mu,x-\varepsilon \lambda}=\bar{X}_{t_{n+1}}^{a,t_{n},x}-\varepsilon
\lambda$, for each $\beta=(a,\mu,\lambda)\in B$, we obtain that for
$n=0,\ldots,N-1,$
\[
\bar{u}_{\Delta}(t_{n}-\varepsilon^{2}\mu,x-\varepsilon \lambda)\geq \sup_{a\in
A}\mathbb{\hat{E}}_{t_{n},x}^{a}\left[  \bar{u}_{\Delta}(t_{n+1}%
-\varepsilon^{2}\mu,\bar{X}_{t_{n+1}}-\varepsilon \lambda)+f(t_{n}%
,x,a)\Delta \right]  .
\]
Taking convolutions of both sides of the above equality with $\zeta
_{\varepsilon}$, from the convexity of $\sup_{a\in A}\mathbb{\hat{E}}[\cdot]$,
we obtain the desired result.
\end{proof}
\end{proposition}

\subsection{Discrete approximation scheme}

According to Theorem \ref{E repesent}, we rewrite the Euler scheme
(\ref{DP v^bar_h}) as%
\begin{equation}%
\begin{split}
&  \bar{v}_{\Delta}(T,x)=g(x),\\
&  \bar{v}_{\Delta}(t_{n},x)=\sup_{a\in A}\sup_{\theta_{\cdot}\in \Sigma
}\mathbb{E}_{t_{n},x}^{a,\theta}\left[  \bar{v}_{\Delta}(t_{n+1},\bar
{X}_{t_{n+1}})+f(t_{n},x,a)\Delta \right]  ,
\end{split}
\label{9}%
\end{equation}
for $n=N-1,\ldots,0$, where $\bar{X}_{\cdot}^{a,\theta,t_{n},x}$ satisfies
\[
\bar{X}_{t_{n+1}}=\bar{X}_{t_{n}}+b(t_{n},\bar{X}_{t_{n}},a)\Delta
+\sigma(t_{n},\bar{X}_{t_{n}},a)\Delta B^{\theta}+h(t_{n},\bar{X}_{t_{n}%
},a)\Delta \langle B\rangle^{\theta},
\]
with $\Delta B^{\theta}:=\int_{t_{n}}^{t_{n+1}}\theta_{s}dW_{s}$ and
$\Delta \langle B\rangle^{\theta}:=\int_{t_{n}}^{t_{n+1}}\theta_{s}^{2}ds$.

We now consider a general discrete-time approximation scheme to further
calculate the value function $\bar{v}_{\Delta}$. Let $\Theta$ be a collection
of $\mathbb{R}$-valued random variables, given perhaps on different
probability spaces such that $\mathbb{E}[\xi]=0,$ $\forall$ $\xi \in \Theta$
and
\[
\inf \limits_{\xi \in \Theta}\mathbb{E}[\xi^{2}]=\underline{\sigma}^{2}\text{,
\  \  \ }\sup \limits_{\xi \in \Theta}\mathbb{E}[\xi^{2}]=\bar{\sigma}^{2}.
\]
Let $\{ \xi_{n}\}_{n=0}^{\infty}\in \Theta$ be a sequence of i.i.d.
$\mathbb{R}$-valued random variables. For any $\alpha=(a_{0},\ldots
,a_{N-1})\in \mathcal{A}_{h}$, we denote by $\hat{X}_{\cdot}^{a,\xi,t_{n},x}$
the discrete approximation of the process $\bar{X}_{\cdot}^{a,\theta,t_{n},x}$
recursively defined by%
\[
\hat{X}_{t_{i+1}}=\hat{X}_{t_{i}}+b(t_{i},\hat{X}_{t_{i}},a)\Delta
+\sigma(t_{i},\hat{X}_{t_{i}},a)\sqrt{\Delta}\xi_{i}+h(t_{i},\hat{X}_{t_{i}%
},a)\Delta(\xi_{i})^{2},
\]
for $i=n,\ldots,N-1$, with $\hat{X}_{t_{n}}=x$. Then define the following
discrete-time approximation scheme:
\begin{equation}%
\begin{split}
&  \hat{v}_{\Delta}(T,x)=g(x),\\
&  \hat{v}_{\Delta}(t_{n},x)=\sup_{a\in A}\sup_{\xi \in \Theta}\mathbb{E}%
_{t_{n},x}^{a,\xi}\left[  \hat{v}_{\Delta}(t_{n+1},\hat{X}_{t_{n+1}}%
)+f(t_{n},x,a)\Delta \right]  .
\end{split}
\label{10}%
\end{equation}
Moreover, we derive that for $n=0,1,\ldots,N-1$ and $x\in \mathbb{R}^{m}$,%

\[
\hat{v}_{\Delta}(t_{n},x)=\sup_{\alpha \in \mathcal{A}_{h}}\sup_{\xi \in \Theta
}\mathbb{E}_{t_{n},x}^{\alpha,\xi}\bigg[g(\hat{X}_{T})+\sum_{i=n}^{N-1}%
f(t_{i},\hat{X}_{t_{i}},a_{i})\Delta \bigg].
\]

\begin{remark}
Indeed, the equation \eqref{10} is a general discrete scheme for the $G$-SOCP
\eqref{1}-\eqref{2}. By choosing the appropriate parameter set, we can derive
different kinds of numerical schemes.
\end{remark}

\begin{example}
Let $\Theta$ be a collection of $\mathbb{R}$-valued random variables, such
that for any $\xi \in \Theta$,
\[
P(\xi=p_{i})=\omega_{i}^{\sigma},\; \; \text{for }\sigma \in \{ \underline
{\sigma},\bar{\sigma}\},\;i=1,2,3,
\]
where%
\[
\left \{
\begin{array}
[c]{ll}%
p_{1}=-1, & \omega_{1}^{\sigma}=\sigma^{2}/2;\\
p_{2}=0, & \omega_{2}^{\sigma}=1-\sigma^{2};\\
p_{3}=1, & \omega_{3}^{\sigma}=\sigma^{2}/2.
\end{array}
\right.
\]
It is easy to verify that for any $\xi \in \Theta$, $k\in \mathbb{N}^{+}$,
$\mathbb{E}[\xi^{2k-1}]=0$,
\[
\inf \limits_{\xi \in \Theta}\mathbb{E}[\xi^{2k}]=\underline{\sigma}^{2}\text{,
\  \  \ }\sup \limits_{\xi \in \Theta}\mathbb{E}[\xi^{2k}]=\bar{\sigma}^{2}.
\]
From the scheme (\ref{10}), we can define the following discrete approximation
scheme
\begin{equation}%
\begin{split}
&  \hat{v}_{\Delta}(T,x)=g(x),\\
&  \hat{v}_{\Delta}(t_{n},x)=\sup_{a\in A}\sup_{\sigma \in \{ \underline{\sigma
},\bar{\sigma}\}}\left[  \sum \limits_{i=1}^{3}\omega_{i}^{\sigma}\hat
{v}_{\Delta}(t_{n+1},\zeta_{t_{n+1}}^{a,i})+f(t_{n},x,a)\Delta \right]  ,
\end{split}
\label{Trinomial}%
\end{equation}
for $n=N-1,\ldots,0$, where
\[
\zeta_{t_{n+1}}^{a,i}=x+b(t_{n},x,a)\Delta+\sigma(t_{n},x,a)\sqrt{\Delta}%
p_{i}+h(t_{n},x,a)\Delta(p_{i})^{2}.
\]

\end{example}

\begin{example}
Let $\Theta$ be a collection of $\mathbb{R}$-valued random variables, such
that for given integer $L\geq2$ and any $\xi \in \Theta,$
\[
P(\xi=p_{i}^{\sigma})=\omega_{i},\; \; \text{for }\sigma \in \{ \underline
{\sigma},\bar{\sigma}\},\;i=1,\ldots,L,
\]
where $\omega_{i}=\frac{A_{i}}{\sqrt{\pi}}$, $p_{i}^{\sigma}=\sigma \sqrt
{2}x_{i}$, and $\{(A_{i},x_{i})\}_{i=1}^{L}\ $are the weights and roots of the
Gauss-Hermite quadrature rule (see e.g.,\  \cite{AS1972}). Noting that the
Gauss-Hermite quadrature rule is exact for any polynomial of degree $2L-1$,
one can check that for any $\xi \in \Theta$, $k=1,2,\ldots,L$, $\mathbb{E}%
[\xi^{2k-1}]=0$,
\[
\inf \limits_{\xi \in \Theta}\mathbb{E}[\xi^{2k}]=\underline{\sigma}%
^{2k}(2k-1)!!\text{, \ }\sup \limits_{\xi \in \Theta}\mathbb{E}[\xi^{2k}%
]=\bar{\sigma}^{2k}(2k-1)!!.
\]
By means of the scheme (\ref{10}), we obtain the following discrete
approximation scheme
\begin{equation}%
\begin{split}
&  \hat{v}_{\Delta}(T,x)=g(x),\\
&  \hat{v}_{\Delta}(t_{n},x)=\sup_{a\in A}\sup_{\sigma \in \{ \underline{\sigma
},\bar{\sigma}\}}\left[  \sum \limits_{i=1}^{L}\omega_{i}\hat{v}_{\Delta
}(t_{n+1},\zeta_{t_{n+1}}^{\sigma,a,i})+f(t_{n},x,a)\Delta \right]  ,
\end{split}
\label{GH}%
\end{equation}
for $n=N-1,\ldots,0$, where
\[
\zeta_{t_{n+1}}^{\sigma,a,i}=x+b(t_{n},x,a)\Delta+\sigma(t_{n},x,a)\sqrt
{\Delta}p_{i}^{\sigma}+h(t_{n},x,a)\Delta(p_{i}^{\sigma})^{2}.
\]

\end{example}

Consider \textquotedblleft shaken" coefficients and use again the argument
defined above. Following the definitions of $b(t,x,\beta)$, $\sigma
(t,x,\beta)$, and $f(t,x,\beta)$, we define the \textquotedblleft shaken"
value function as follows:%

\[
\hat{u}_{\Delta}(t_{n},x)=\sup_{\beta \in \mathcal{B}_{h}}\sup_{\xi \in \Theta
}\mathbb{E}_{t_{n},x}^{\beta,\xi}\bigg[g(\hat{X}_{t_{N_{s}}})+\sum
_{i=n}^{N_{s}}f(t_{i},\hat{X}_{t_{i}},\beta_{i})\Delta \bigg],
\]
where $\hat{X}_{\cdot}^{\beta,\xi,t_{n},x}$ is recursively defined by
\[
\hat{X}_{t_{i+1}}=\hat{X}_{t_{i}}+b(t_{i},\hat{X}_{t_{i}},\beta_{i}%
)\Delta+\sigma(t_{i},\hat{X}_{t_{i}},\beta_{i})\sqrt{\Delta}\xi_{i}%
+h(t_{i},\hat{X}_{t_{i}},\beta_{i})\Delta(\xi_{i})^{2},
\]
for $i=n,\ldots,N_{s}-1$ with $\hat{X}_{t_{n}}=x$.

We also extend $\hat{v}_{\Delta}$ and $\hat{u}_{\Delta}\ $defined on the grid
points to the whole of $[0,T]$ and $[0,S]$, respectively. In the similar way
above, we give the following properties.

\begin{proposition}
\label{prop u^}Assume that \emph{(H1)-(H2)} hold. Then, there exists a
constant $C\geq0$ such that for any $t\in \lbrack0,T]$, $x\in \mathbb{R}^{m}$,
\[
|\hat{u}_{\Delta}(t,x)-\hat{v}_{\Delta}(t,x)|\leq C\varepsilon,
\]
and for any $t,s\in \lbrack0,S]$, $x,y\in \mathbb{R}^{m}$
\[
|\hat{u}_{\Delta}(t,x)-\hat{u}_{\Delta}(s,y)|\leq C(|t-s|^{\frac{1}{2}}%
+\Delta^{\frac{1}{2}}+|x-y|).
\]

\end{proposition}

\begin{proposition}
\label{prop u^epsi^}Assume that \emph{(H1)-(H2)} hold. Then,\

\begin{description}
\item[(i)] there exists a constant $C\geq0$ such that for any $t\in
\lbrack0,S]$, $x\in \mathbb{R}^{m},$
\[
|\hat{u}_{\Delta}(t,x)-\hat{u}_{\Delta}^{(\varepsilon)}(t,x)|\leq
C\varepsilon;
\]

\item[(ii)] the function $\hat{u}_{\Delta}^{(\varepsilon)}\in C_{0}^{\infty
}([0,S]\times \mathbb{R}^{m})$ and
\[
\left \Vert \partial_{t}^{l}D_{x}^{k}\hat{u}_{\Delta}^{(\varepsilon
)}\right \Vert _{\infty}\leq C\varepsilon^{1-2l-k}\text{, \ for }k+l\geq1;
\]

\item[(iii)] $\hat{u}_{\Delta}^{(\varepsilon)}$ satisfies the following
super-dynamic programming principle: for any $n=0,\ldots,N-1$, $x\in
\mathbb{R}^{m}$,
\[
\hat{u}_{\Delta}^{(\varepsilon)}(t_{n},x)\geq \sup_{a\in A}\sup_{\xi \in \Theta
}\mathbb{E}_{t_{n},x}^{a,\xi}\left[  \hat{u}_{\Delta}^{(\varepsilon)}%
(t_{n+1},\hat{X}_{t_{n+1}})+f(t_{n},x,a)\Delta \right]  .
\]

\end{description}
\end{proposition}

\section{Convergence rate of the discrete approximation scheme}

In this section, we derive an error bound for the convergence of the solution
of our discrete approximation scheme (\ref{10}). In the sequel, $C$ represents
a generic constant which does not depend on the time partition and may be
different from line to line. We also need the following assumption:

\begin{description}
\item[(H3)] For any $\xi \in \Theta$, $\mathbb{E}[|\xi|^{3}]<\infty$.
\end{description}

The main result in this paper is the following:

\begin{theorem}
\label{main theorem}Assume \emph{(H1)-(H3)} hold. Let $v$ and $\hat{v}
_{\Delta}$ be the solution of (\ref{2}) and (\ref{10}), respectively. Then,

\begin{description}
\item[(i)] for any $x\in \mathbb{R}^{m}$, $n=0,1,\ldots,N$,
\[
|v(t_{n},x)-\hat{v}_{\Delta}(t_{n},x)|\leq C\Delta^{\frac{1}{6}};
\]

\item[(ii)] moreover, if $\mathbb{E}[\xi^{3}]=0$ for any $\xi \in \Theta$ and
$\sup \limits_{\xi \in \Theta}\mathbb{E}[|\xi|^{4}]<\infty$, then for any
$x\in \mathbb{R}^{m}$, $n=0,1,\ldots,N$,
\[
|v(t_{n},x)-\hat{v}_{\Delta}(t_{n},x)|\leq C\Delta^{\frac{1}{4}}.
\]

\end{description}
\end{theorem}

\begin{proof}
Without loss of generality, we suppose that $m=1$ and $h=0$. (i) Step 1. Lower
bound on $v-\hat{v}_{\Delta}$. From Proposition \ref{prop u^epsi-} (iii) and
Theorem \ref{E repesent}, we have for $n=0,\ldots,N-1$,%
\begin{equation}
\bar{u}_{\Delta}^{(\varepsilon)}(t_{n},x)\geq \sup_{a\in A}\sup_{\theta
\in \lbrack \underline{\sigma},\bar{\sigma}]}\mathbb{E}_{t_{n},x}^{a,\theta
}\left[  \bar{u}_{\Delta}^{(\varepsilon)}(t_{n+1},\bar{X}_{t_{n+1}}%
)+f(t_{n},x,a)\Delta \right]  . \label{3}%
\end{equation}
Denote%
\[
L_{a}^{\lambda}u(t,x)=\partial_{t}u(t,x)+b(t,x,a)\partial_{x}u(t,x)+\frac
{1}{2}\lambda^{2}\sigma^{2}(t,x,a)\partial_{x}^{2}u(t,x).
\]
Using Taylor's formula and Proposition \ref{Prop Ito's integral}, we have
\begin{equation}
\mathbb{E}_{t_{n},x}^{a,\theta}\left[  \bar{u}_{\Delta}^{(\varepsilon
)}(t_{n+1},\bar{X}_{t_{n+1}})\right]  =\bar{u}_{\Delta}^{(\varepsilon)}%
(t_{n},x)+L_{a}^{\theta}\bar{u}_{\Delta}^{(\varepsilon)}(t_{n},x)\Delta
+\bar{R}_{\Delta}^{\varepsilon}, \label{4}%
\end{equation}
where%
\begin{align*}
\bar{R}_{\Delta}^{\varepsilon}  &  =\frac{1}{2}\Delta^{2}\left[  \partial
_{t}^{2}\bar{u}_{\Delta}^{(\varepsilon)}(t_{n},x)+\partial_{x}^{2}\bar
{u}_{\Delta}^{(\varepsilon)}(t_{n},x)b^{2}(t_{n},x,a)+2\partial_{tx}^{2}%
\bar{u}_{\Delta}^{(\varepsilon)}(t_{n},x)b(t_{n},x,a)\right] \\
&  \text{ \  \ }+\frac{1}{3!}\mathbb{E}\left[  (\Delta \partial_{t}+\Delta
\bar{X}\partial_{x})^{3}\bar{u}_{\Delta}^{(\varepsilon)}(t_{n}+\tau
\Delta,x+\tau \Delta \bar{X})\right]  ,
\end{align*}
with $\Delta \bar{X}:=\bar{X}_{t_{n+1}}^{a,\theta,t_{n},x}-x$ and $\tau
\in \lbrack0,1]$. By Proposition \ref{prop u^epsi-} (ii), one can check that%
\begin{equation}
|\bar{R}_{\Delta}^{\varepsilon}|\leq C\sum \limits_{i=0}^{3}\varepsilon
^{-2-i}\Delta^{\frac{i+3}{2}}. \label{5}%
\end{equation}
Together with (\ref{3}), (\ref{4}), and (\ref{5}), we obtain that for
$n=0,\ldots,N-1$,
\begin{equation}
\sup_{a\in A}\sup_{\theta \in \lbrack \underline{\sigma},\bar{\sigma}]}\left[
L_{a}^{\theta}\bar{u}_{\Delta}^{(\varepsilon)}(t_{n},x)+f(t_{n},x,a)\right]
\leq C\sum \limits_{i=0}^{3}\varepsilon^{-2-i}\Delta^{\frac{i+1}{2}}. \label{6}%
\end{equation}

We claim that if $1\leq k\leq N$ and $\eta:=\sup \limits_{x\in \mathbb{R}}%
[\hat{v}_{\Delta}(t_{k},x)-\bar{u}_{\Delta}^{(\varepsilon)}(t_{k},x)]$, then
for any $x\in \mathbb{R}$, $n\leq k$,
\begin{equation}
\hat{v}_{\Delta}(t_{n},x)\leq \bar{u}_{\Delta}^{(\varepsilon)}(t_{n}%
,x)+\eta+\bar{I}_{\Delta}^{\varepsilon}(t_{k}-t_{n}), \label{assert}%
\end{equation}
where%
\[
\bar{I}_{\Delta}^{\varepsilon}=C\sum \limits_{i=0}^{3}\varepsilon^{-2-i}%
\Delta^{\frac{i+1}{2}}.
\]
Taking $k=N$ in (\ref{assert}) and using the following estimates that by
Propositions \ref{prop u^-}-\ref{prop u^epsi-}
\[%
\begin{array}
[c]{l}%
|\bar{u}_{\Delta}^{(\varepsilon)}(T,x)-\bar{u}_{\Delta}(T,x)|+|\bar{u}%
_{\Delta}(T,x)-g(x)|\leq C(|S-T|^{1/2}+\varepsilon)\leq C\varepsilon,\\
|\bar{u}_{\Delta}^{(\varepsilon)}(t_{n},x)-\bar{u}_{\Delta}(t_{n},x)|+|\bar
{u}_{\Delta}(t_{n},x)-\bar{v}_{\Delta}(t_{n},x)|\leq C\varepsilon,
\end{array}
\]
we can deduce that for any $x\in \mathbb{R}$, $n=0,\ldots,N$,
\begin{equation}
\hat{v}_{\Delta}(t_{n},x)\leq \bar{v}_{\Delta}(t_{n},x)+\bar{I}_{\Delta
}^{\varepsilon}(T-t_{n})+C\varepsilon \leq \bar{v}_{\Delta}(t_{n},x)+\bar
{J}_{\Delta}^{\varepsilon}, \label{lower bound}%
\end{equation}
where
\[
\bar{J}_{\Delta}^{\varepsilon}=C(\varepsilon+\sum \limits_{i=0}^{3}%
\varepsilon^{-2-i}\Delta^{\frac{i+1}{2}}).
\]

Now we prove the assertion (\ref{assert}) by induction. For $n=k$, the
inequality (\ref{assert}) obviously holds. Assume that for $1\leq n\leq k$,
the assertion (\ref{assert}) holds.
Then
\begin{align}
\hat{v}_{\Delta}(t_{n-1},x)  &  =\sup_{a\in A}\sup_{\xi \in \Theta}%
\mathbb{E}_{t_{n-1},x}^{a,\xi}\left[  \hat{v}_{\Delta}(t_{n},\hat{X}_{t_{n}%
})+f(t_{n-1},x,a)\Delta \right] \label{7}\\
&  \leq \sup_{a\in A}\sup_{\xi \in \Theta}\mathbb{E}_{t_{n-1},x}^{a,\xi}\left[
\bar{u}_{\Delta}^{(\varepsilon)}(t_{n},\hat{X}_{t_{n}})+f(t_{n-1}%
,x,a)\Delta \right]  +\eta+\bar{I}_{\Delta}^{\varepsilon}(t_{k}-t_{n}%
).\nonumber
\end{align}
By Taylor's formula, for any $\xi \in \Theta$, we can deduce
\begin{equation}
\mathbb{E}_{t_{n-1},x}^{a,\xi}\left[  \bar{u}_{\Delta}^{(\varepsilon)}%
(t_{n},\hat{X}_{t_{n}})\right]  =\bar{u}_{\Delta}^{(\varepsilon)}%
(t_{n-1},x)+\mathbb{E}[L_{a}^{\xi}\bar{u}_{\Delta}^{(\varepsilon)}%
(t_{n-1},x)\Delta]+\bar{H}_{\Delta}^{\varepsilon}, \label{taylor}%
\end{equation}
where%
\begin{align*}
\bar{H}_{\Delta}^{\varepsilon}  &  =\Delta^{2}\left[  \partial_{tx}^{2}\bar
{u}_{\Delta}^{(\varepsilon)}(t_{n-1},x)b(t_{n-1},x,a)\right] \\
&  \text{ \  \ }+\frac{1}{2}\Delta^{2}\left[  \partial_{t}^{2}\bar{u}_{\Delta
}^{(\varepsilon)}(t_{n-1},x)+\partial_{x}^{2}\bar{u}_{\Delta}^{(\varepsilon
)}(t_{n-1},x)b^{2}(t_{n-1},x,a)\right] \\
&  \text{ \  \ }+\frac{1}{3!}\mathbb{E}\left[  (\Delta \partial_{t}+\Delta
\hat{X}\partial_{x})^{3}\bar{u}_{\Delta}^{(\varepsilon)}(t_{n-1}+\tau
\Delta,x+\tau \Delta \hat{X})\right]  ,
\end{align*}
with $\Delta \hat{X}:=\hat{X}_{t_{n}}^{a,\xi,t_{n-1},x}-x$ and $\tau \in(0,1)$.
Under the assumption (H3), from the regularity of $\partial_{t}^{l}D_{x}%
^{k}\bar{u}_{\Delta}^{(\varepsilon)}$ given in Proposition \ref{prop u^epsi-}
(ii), it follows that
\[
|\bar{H}_{\Delta}^{\varepsilon}|\leq C\sum \limits_{i=0}^{3}\varepsilon
^{-2-i}\Delta^{\frac{i+3}{2}}.
\]
Together with (\ref{6}), (\ref{7}), and (\ref{taylor}), we can derive that
\begin{align*}
\hat{v}_{\Delta}(t_{n-1},x)  &  \leq \bar{u}_{\Delta}^{(\varepsilon)}
(t_{n-1},x)+\eta+\bar{I}_{\Delta}^{\varepsilon}(t_{k}-t_{n})+2|\bar{H}%
_{\Delta}^{\varepsilon}|\\
&  \leq \bar{u}_{\Delta}^{(\varepsilon)}(t_{n-1},x)+\eta+\bar{I}_{\Delta
}^{\varepsilon}(t_{k}-t_{n-1}),
\end{align*}
where we have used the fact that
\[
\sup_{\xi \in \Theta}\mathbb{E}\left[  L_{a}^{\xi}\bar{u}_{\Delta}%
^{(\varepsilon)}(t_{n-1},x)\right]  =\sup_{\theta \in \lbrack \underline{\sigma
},\bar{\sigma}]}L_{a}^{\theta}\bar{u}_{\Delta}^{(\varepsilon)}(t_{n-1},x).
\]
By the principle of induction the assertion is true for all $0\leq n\leq k$
and $x\in \mathbb{R}$.

Step 2. Upper bound on $v(t,x)-\hat{v}_{\Delta}(t,x)$. From Proposition
\ref{prop u^epsi^} (iii), by a similar analysis as step 1, we obtain that for
$n=0,\ldots,N-1$,
\begin{equation}
\sup_{\theta_{\cdot}\in \Sigma}\mathbb{E}\left[  L_{a}^{\theta}\hat{u}_{\Delta
}^{(\varepsilon)}(t_{n},x)+f(t_{n},x,a)\right]  =\sup_{\xi \in \Theta}%
\mathbb{E}\left[  L_{a}^{\xi}\hat{u}_{\Delta}^{(\varepsilon)}(t_{n}%
,x)+f(t_{n},x,a)\right]  \leq C\sum \limits_{i=0}^{3}\varepsilon^{-2-i}%
\Delta^{\frac{i+1}{2}}. \label{13}%
\end{equation}
Similarly, by induction, it follows from (\ref{13}) that if $1\leq k\leq N$
and $\hat{\eta}:=\sup \limits_{x\in \mathbb{R}}[\bar{v}_{\Delta}(t_{k}%
,x)-\hat{u}_{\Delta}^{(\varepsilon)}(t_{k},x)]$, then for any $x\in \mathbb{R}%
$, $n\leq k$,%
\begin{equation}
\bar{v}_{\Delta}(t_{n},x)\leq \hat{u}_{\Delta}^{(\varepsilon)}(t_{n}%
,x)+\hat{\eta}+\hat{I}_{\Delta}^{\varepsilon}(t_{k}-t_{n}), \label{14}%
\end{equation}
where $\hat{I}_{\Delta}^{\varepsilon}=C\sum \limits_{i=0}^{3}\varepsilon
^{-2-i}\Delta^{\frac{i+1}{2}}$. Taking $k=N$ in (\ref{14}), by Propositions
\ref{prop u^}-\ref{prop u^epsi^}, one knows that%
\[%
\begin{array}
[c]{l}%
|\hat{u}_{\Delta}^{(\varepsilon)}(T,x)-\hat{u}_{\Delta}(T,x)|+|\hat{u}%
_{\Delta}(T,x)-g(x)|\leq C(|S-T|^{1/2}+\varepsilon)\leq C\varepsilon,\\
|\hat{u}_{\Delta}^{(\varepsilon)}(t_{n},x)-\hat{u}_{\Delta}(t_{n},x)|+|\hat
{u}_{\Delta}(t_{n},x)-\hat{v}_{\Delta}(t_{n},x)|\leq C\varepsilon.
\end{array}
\]
Thus, we have for any $n=0,\ldots,N$, $x\in \mathbb{R}$,
\begin{equation}
\bar{v}_{\Delta}(t_{n},x)\leq \hat{v}_{\Delta}(t_{n},x)+C(\varepsilon
+\sum \limits_{i=0}^{3}\varepsilon^{-2-i}\Delta^{\frac{i+1}{2}} ).
\label{upper bound}%
\end{equation}

Step 3. Conclusion. Together with (\ref{lower bound}) and (\ref{upper bound}),
by taking $\varepsilon=\Delta^{1/6}$, then we obtain that for any
$n=0,\ldots,N$, $x\in \mathbb{R}$,
\begin{equation}
|\bar{v}_{\Delta}(t_{n},x)-\hat{v}_{\Delta}(t_{n},x)|\leq C\Delta^{\frac{1}%
{6}}; \label{17}%
\end{equation}
Combining (\ref{Euler approxi}) and (\ref{17}), we conclude (i).

(ii) We follow the proof of (i) and take more terms in Taylor's formula. Then
the remainder $\bar{R}_{\Delta}^{\varepsilon}$ in (\ref{4}) takes the form
\[%
\begin{split}
\bar{R}_{\Delta}^{\varepsilon}  &  =\frac{1}{2}\Delta^{2}\left[  \partial
_{t}^{2}\bar{u}_{\Delta}^{(\varepsilon)}(t_{n},x)+\partial_{x}^{2}\bar
{u}_{\Delta}^{(\varepsilon)}(t_{n},x)b^{2}(t_{n},x,a)+2\partial_{tx}^{2}%
\bar{u}_{\Delta}^{(\varepsilon)}(t_{n},x)b(t_{n},x,a)\right] \\
&  \text{ \  \ }+\frac{1}{2}\Delta^{2}\left[  \partial_{x}^{3}\bar{u}_{\Delta
}^{(\varepsilon)}(t_{n},x)b(t_{n},x,a)\sigma^{2}(t_{n},x,a)+\partial_{txx}%
^{3}\bar{u}_{\Delta}^{(\varepsilon)}(t_{n},x)\sigma^{2}(t_{n},x,a)\right]
\mathbb{\theta}^{2}\\
&  \text{ \  \ }+\frac{1}{2}\Delta^{3}\left[  \partial_{ttx}^{3}\bar{u}%
_{\Delta}^{(\varepsilon)}(t_{n},x)b(t_{n},x,a)+\partial_{txx}^{3}\bar
{u}_{\Delta}^{(\varepsilon)}(t_{n},x)b^{2}(t_{n},x,a)\right] \\
&  \text{ \  \ }+\frac{1}{6}\Delta^{3}\left[  \partial_{t}^{3}\bar{u}_{\Delta
}^{(\varepsilon)}(t_{n},x)+\partial_{x}^{3}\bar{u}_{\Delta}^{(\varepsilon
)}(t_{n},x)b^{3}(t_{n},x,a)\right] \\
&  \text{ \  \ }+\frac{1}{4!}\mathbb{E}\left[  (\Delta \partial_{t}+\Delta
\bar{X}\partial_{x})^{4}\bar{u}_{\Delta}^{(\varepsilon)}(t_{n}+\tau
\Delta,x+\tau \Delta \bar{X})\right]  ,
\end{split}
\label{19}%
\]
with $\Delta \bar{X}=\bar{X}_{t_{n+1}}^{a,\theta,t_{n},x}-x$ and $\tau
\in \lbrack0,1]$. In addition, take again more terms in Taylor's formula, then
the remainder $\bar{H}_{\Delta}^{\varepsilon}$ in (\ref{taylor}) can be
expressed as
\[%
\begin{split}
\bar{H}_{\Delta}^{\varepsilon}  &  =\Delta^{2}\left[  \partial_{tx}^{2}\bar
{u}_{\Delta}^{(\varepsilon)}(t_{n-1},x)b(t_{n-1},x,a)\right] \\
&  \text{ \  \ }+\frac{1}{2}\Delta^{2}\left[  \partial_{t}^{2}\bar{u}_{\Delta
}^{(\varepsilon)}(t_{n-1},x)+\partial_{x}^{2}\bar{u}_{\Delta}^{(\varepsilon
)}(t_{n-1},x)b^{2}(t_{n-1},x,a)\right] \\
&  \text{ \  \ }+\frac{1}{2}\Delta^{2}\left[  \partial_{txx}^{3}\bar{u}%
_{\Delta}^{(\varepsilon)}(t_{n-1},x)\sigma^{2}(t_{n-1},x,a)\mathbb{E}[\xi
^{2}]\right] \\
&  \text{ \  \ }+\frac{1}{2}\Delta^{2}\left[  \partial_{x}^{3}\bar{u}_{\Delta
}^{(\varepsilon)}(t_{n-1},x)b(t_{n-1},x,a)\sigma^{2}(t_{n-1},x,a)\mathbb{E}%
[\xi^{2}]\right] \\
&  \text{ \  \ }+\frac{1}{2}\Delta^{3}\left[  \partial_{ttx}^{3}\bar{u}%
_{\Delta}^{(\varepsilon)}(t_{n-1},x)b(t_{n-1},x,a)+\partial_{txx}^{3}\bar
{u}_{\Delta}^{(\varepsilon)}(t_{n-1},x)b^{2}(t_{n-1},x,a)\right] \\
&  \text{ \  \ }+\frac{1}{6}\Delta^{3}\left[  \partial_{t}^{3}\bar{u}_{\Delta
}^{(\varepsilon)}(t_{n-1},x)+\partial_{x}^{3}\bar{u}_{\Delta}^{(\varepsilon
)}(t_{n-1},x)b^{3}(t_{n-1},x,a)\right] \\
&  \text{ \  \ }+\frac{1}{6}\Delta^{\frac{3}{2}}\left[  \partial_{x}^{3}\bar
{u}_{\Delta}^{(\varepsilon)}(t_{n-1},x)\sigma^{3}(t_{n-1},x,a)\mathbb{E}%
[\xi^{3}]\right] \\
&  \text{ \  \ }+\frac{1}{4!}\mathbb{E}\left[  (\Delta \partial_{t}+\Delta
\hat{X}\partial_{x})^{4}\bar{u}_{\Delta}^{(\varepsilon)}(t_{n-1}+\tau
\Delta,x+\tau \Delta \hat{X})\right]  ,
\end{split}
\label{20}%
\]
where $\Delta \hat{X}=\hat{X}_{t_{n}}^{a,\xi,t_{n-1},x}-x$, $\tau \in(0,1)$.
Seeing that $\mathbb{E}[\xi^{3}]=0$, $\forall \xi \in \Theta$ and $\sup
\limits_{\xi \in \Theta}\mathbb{E}[|\xi|^{4}]<\infty$, by Proposition
\ref{prop u^epsi-} (ii), we deduce that $|\bar{H}_{\Delta}^{\varepsilon}|\leq
C\sum_{i=1}^{5}\varepsilon^{-2-i}\Delta^{\frac{i+3}{2}}$. This indicates that
for any $n=0,\ldots,N$, $x\in \mathbb{R}$,
\[
\hat{v}_{\Delta}(t_{n},x)\leq \bar{v}_{\Delta}(t_{n},x)+C(\varepsilon
+\sum \limits_{i=1}^{5}\varepsilon^{-2-i}\Delta^{\frac{i+1}{2}}).
\]
Analogously, we can also obtain the upper bound, for any $n=0,\ldots,N$,
$x\in \mathbb{R}$,
\[
\bar{v}_{\Delta}(t_{n},x)\leq \hat{v}_{\Delta}(t_{n},x)+C(\varepsilon
+\sum \limits_{i=1}^{5}\varepsilon^{-2-i}\Delta^{\frac{i+1}{2}}).
\]
Taking $\varepsilon=\Delta^{1/4}$, by (\ref{Euler approxi}), we conclude the proof.
\end{proof}

\begin{remark}
According to Theorem \ref{main theorem}, we know that the discrete schemes
\eqref{Trinomial} and \eqref{GH} admit a $1/4$ order rate of convergence.
\end{remark}


\section{Numerical example}

In this section, we present several numerical simulations to illustrate the
obtained theoretical results. In the following, we apply the scheme
(\ref{Trinomial}) and scheme (\ref{GH}) with $L=6$ to test the examples. In
our tables, we shall denote by CR the convergence rate, TR the Scheme
(\ref{Trinomial}), and GH the Scheme (\ref{GH}), respectively.

\begin{example}
\label{EX3}We first consider the following $G$-heat equation%
\[
\left \{
\begin{array}
[c]{l}%
\partial_{t}v+G\big(\partial_{xx}^{2}v\big)=0,\text{\  \ }(t,x)\in
(0,T]\times \mathbb{R},\\
v(T,x)=g(x),
\end{array}
\right.
\]
where
\[
g(x)=\left \{
\begin{array}
[c]{ll}%
\frac{2}{1+\beta}\cos \big(\frac{1+\beta}{2}x\big), & x\in \big[-\frac{\pi
}{1+\beta}+2k\pi,\frac{\pi}{1+\beta}+2k\pi \big),\\
\frac{2\beta}{1+\beta}\cos \big(\frac{1+\beta}{2\beta}x+\frac{\beta-1}{2\beta
}\pi \big), & x\in \big[\frac{\pi}{1+\beta}+2k\pi,\frac{(2\beta+1)\pi}{1+\beta
}+2k\pi \big),
\end{array}
\right.  k\in \mathbb{Z}\text{,}%
\]
with $\beta=\frac{\overline{\sigma}}{\underline{\sigma}}$. Set $T=1$,
$x_{0}=0$, $\underline{\sigma}=0.1$ and $\overline{\sigma}=1$. The exact
solution of the $G$-heat equation$\ $given in \cite{Song2015} is $v\left(
t,x\right)  =e^{-\frac{\rho^{2}(1-t)}{2}}g\left(  x\right)  $, where
$\rho=\frac{\underline{\sigma}+\overline{\sigma}}{2}$. We solve this example
by the scheme \eqref{Trinomial} and scheme \eqref{GH}, and the errors and
convergence rates are listed in Table \ref{Table 3}. In our test, the value
function and data are often more regular than the assumption, which leads to a
higher convergence rate than the theoretical result. \begin{table}[tbh]
\caption{Errors and convergence rates for Example \ref{EX3}.}%
\label{Table 3}
{\footnotesize
\[%
\begin{tabular}
[c]{||c|c|c|c|c|c|c||}\hline
\multicolumn{7}{||c||}{$|v(0,x_{0})-\hat{v}_{\Delta}(0,x_{0})|$}\\ \hline
Scheme & $N=16$ & $N=32$ & $N=64$ & $N=128$ & $N=256$ & CR\\ \hline
TR & 3.217E-03 & 1.684E-03 & 8.693E-04 & 4.405E-04 & 2.226E-04 & 0.964\\ \hline
GH & 2.605E-03 & 1.331E-03 & 6.663E-04 & 3.336E-04 & 1.665E-04 & 0.993\\ \hline
\end{tabular}
\]
}\end{table}
\end{example}

\begin{example}
\label{exp1}We now test a fully nonlinear HJB equation%
\[
\left \{
\begin{array}
[c]{l}%
\partial_{t}v+\sup \limits_{a\in A}\left[  G(a^{2}D_{x}^{2}v)+(\kappa
x-a)D_{x}v+2\sqrt{a}e^{-r_{0}t}\right]  =0,\text{ \ }(t,x)\in \lbrack
0,T)\times \mathbb{R},\\
v(T,x)=x,
\end{array}
\right.
\]
where the control set $A=[\frac{1}{5},1]$. This can be interpreted as the
value function of a control problem as follows%
\[
v(t,x)=\sup_{\alpha \in \mathcal{A}}\mathbb{\hat{E}}_{t,x}^{\alpha}%
\bigg[X_{T}+\int_{t}^{T}2\sqrt{\alpha_{s}}e^{-r_{0}s}ds\bigg],
\]
with
\begin{equation}
\left \{
\begin{array}
[c]{ll}%
dX_{s}=(\kappa X_{s}-\alpha_{s})ds+\alpha_{s}dB_{s}, & s\in(t,T],\\
X_{t}=x, &
\end{array}
\right.  \label{ex1}%
\end{equation}
where $B_{1}\sim N(0,[\underline{\sigma}^{2},\bar{\sigma}^{2}])$.
It is easy to verify that the solution of \eqref{ex1} is
\[
X_{s}^{t,x}=e^{\kappa(s-t)}x-\int_{t}^{s}\alpha_{r}e^{\kappa(s-r)}dr+\int
_{t}^{s}\alpha_{r}e^{\kappa(s-r)}dB_{r}.
\]
Then, the optimal control is $\alpha_{t}^{\ast}=e^{2t(\kappa-r_{0})-2\kappa
T}$ and the corresponding value function is
\[
v^{\ast}(t,x)=e^{\kappa(T-t)}x+\frac{e^{-\kappa T}}{\kappa-2r_{0}}\left[
e^{(\kappa-2r_{0})T}-e^{(\kappa-2r_{0})t}\right]  .
\]
We test this example by the scheme \eqref{Trinomial}. The errors and
convergence rates for $T=1$, $x_{0}=0$, $\underline{\sigma}=0.5$, $\bar
{\sigma}=1$, $\kappa=0.5$, and $r_{0}=0.03$ are shown in Table \ref{Table 1}.

\begin{table}[ptb]
\caption{Errors and convergence rates for Example \ref{exp1}.}%
\label{Table 1}
{\footnotesize
\[%
\begin{tabular}
[c]{||c|c|c|c|c|c|c||}\hline
$N$ & $16$ & $32$ & $64$ & $128$ & $256$ & CR\\ \hline
$|v(0,x_{0})-\hat{v}_{\Delta}(0,x_{0})|$ & 1.599E-02 & 8.016E-03 & 4.011E-03 &
2.005E-03 & 1.000E-03 & 1.000\\ \hline
\end{tabular}
\]
}\end{table}
\end{example}

\begin{example}
\label{exp2}Consider the following fully nonlinear HJB equation%
\[
\left \{
\begin{array}
[c]{l}%
\partial_{t}v+\sup \limits_{a\in A}\left[  G(\sin^{2}(t+x)D_{x}^{2}%
v)+(2a\sin^{2}(t+x)-1)D_{x}v\right. \\
\text{ \  \  \ }+\left.  2\cos^{2}(t+x)-\cos^{4}(t+x)-a^{2}\right]  =0,\text{
}(t,x)\in \lbrack0,T)\times \mathbb{R},\\
v(T,x)=x,
\end{array}
\right.
\]
where the control set $A=[0,1]$. One can express the solution of the above
equation as the value function of the following control problem
\[
v(t,x)=\sup_{\alpha \in \mathcal{A}}\mathbb{\hat{E}}_{t,x}^{\alpha}%
\bigg[X_{T}+\int_{t}^{T}\left(  2\cos^{2}(s+X_{s})-\cos^{4}(s+X_{s}%
)-\alpha_{s}^{2}\right)  ds\bigg],
\]
with
\[
\left \{
\begin{array}
[c]{ll}%
dX_{s}=\left(  2\alpha_{s}\sin^{2}(s+X_{s})-1\right)  ds+\sin^{2}%
(s+X_{s})dB_{s}, & s\in(t,T],\\
X_{t}=x, &
\end{array}
\right.
\]
where $B_{1}\sim N(0,[\underline{\sigma}^{2},\bar{\sigma}^{2}])$. It can be
checked that the optimal control is $\alpha_{t}^{\ast}=\sin^{2}(t+X_{t})$\ and
the corresponding value function is $v^{\ast}(t,x)=x$. The numerical results
by using the scheme \eqref{GH} are reported in Table \ref{Table 2} with $T=1$,
$x_{0}=0$, $\underline{\sigma}=0.5$, and $\bar{\sigma}=1$. It can be seen that
the numerical simulation is consistent with our theoretical results.
\begin{table}[tbh]
\caption{Errors and convergence rates for Example \ref{exp2}.}%
\label{Table 2}
{\footnotesize
\[%
\begin{tabular}
[c]{||c|c|c|c|c|c|c||}\hline
$N$ & $16$ & $32$ & $64$ & $128$ & $256$ & CR\\ \hline
$|v(0,x_{0})-\hat{v}_{\Delta}(0,x_{0})|$ & 2.598E-05 & 2.303E-05 & 1.465E-05 &
4.165E-06 & 1.058E-06 & 1.170\\ \hline
\end{tabular}
\]
}\end{table}
\end{example}

\section{Appendix}

\begin{proof}
[Proof of Theorem \ref{lemma1}]Let $\tilde{X}_{\cdot}$ be the solution of
(\ref{1}) with coefficients replaced by $b^{(\varepsilon)},\sigma
^{(\varepsilon)}$\ and $h^{(\varepsilon)}$, and $\tilde{v}$ be the value
function of the optimal control problems (\ref{1})-(\ref{2}) with $X_{\cdot
},f,g$ replaced by $\tilde{X}_{\cdot},f^{(\varepsilon)},g^{(\varepsilon)}$
satisfying the estimate%
\begin{equation}
|v(t,x)-\tilde{v}(t,x)|\leq C\varepsilon,\text{ }\forall \left(  t,x\right)
\in \lbrack0,T]\times \mathbb{R}^{m}. \label{A.0}%
\end{equation}
For simplicity, we still use $(b,\sigma,f,g)$ instead of $(b^{(\varepsilon
)},\sigma^{(\varepsilon)},f^{(\varepsilon)},g^{(\varepsilon)})$. By
Propositions 2.3-2.4 in \cite{JPR2019}, one can find a family of smooth
functions $\omega_{\varepsilon}(t,x)\in C_{0}^{\infty}([0,T]\times
\mathbb{R}^{m})$ with $\varepsilon \in(0,1)$, such that for any $t,s\in
\lbrack0,T]$ and $x,y\in \mathbb{R}^{m},$
\begin{equation}
|\omega_{\varepsilon}(t,x)-\omega_{\varepsilon}(s,y)|\leq C(\varepsilon
+|t-s|^{\frac{1}{2}}+|x-y|), \label{A.0.1}%
\end{equation}
and
\begin{equation}
\left \Vert \omega_{\varepsilon}-v_{\Delta}\right \Vert _{\infty}\leq
C\varepsilon \  \  \text{and \ }\left \Vert \partial_{t}^{l}D_{x}^{k}%
\omega_{\varepsilon}\right \Vert _{\infty}\leq C\varepsilon^{1-2l-k},\text{
\ for }k+l\geq1, \label{A.0.2}%
\end{equation}
where $C$ is a positive constant independent of $\Delta$. Moreover, one can
obtain that, for any $\left(  t,x\right)  \in \lbrack0,T-\Delta]\times
\mathbb{R}^{m},a\in A,$
\begin{equation}
\omega_{\varepsilon}(t,x)\geq \mathbb{\hat{E}}_{t,x}^{a}\bigg[\omega
_{\varepsilon}(t+\Delta,\tilde{X}_{t+\Delta})+\int_{t}^{t+\Delta}f(s,\tilde
{X}_{s},a)ds\bigg]. \label{A.1}%
\end{equation}

Step 1. Upper bound on $L_{a}^{0}\omega_{\varepsilon}+2G(L_{a}^{1}%
\omega_{\varepsilon})+f$.\ For any $\left(  t,x\right)  \in \lbrack
0,T-\Delta]\times \mathbb{R}^{m},a\in A$, applying $G$-It\^{o}'s formula to
$\omega_{\varepsilon}$ (twice) and $f$, we have%
\begin{equation}%
\begin{split}
&  \omega_{\varepsilon}(t+\Delta,\tilde{X}_{t+\Delta})+\int_{t}^{t+\Delta
}f(s,\tilde{X}_{s},a)ds\\
&  =\omega_{\varepsilon}(t,x)+L_{a}^{0}\omega_{\varepsilon}(t,x)\Delta
+L_{a}^{1}\omega_{\varepsilon}(t,x)\Delta \langle B\rangle+f(t,x,a)\Delta \\
&  \text{ \  \ }+\int_{t}^{t+\Delta}\int_{t}^{s}L_{a}^{0}f(r,\tilde{X}%
_{r},a)drds+\int_{t}^{t+\Delta}\int_{t}^{s}L_{a}^{1}f(r,\tilde{X}%
_{r},a)d\langle B\rangle_{r}ds\\
&  \text{ \  \ }+\int_{t}^{t+\Delta}\int_{t}^{s}L_{a}^{0}L_{a}^{0}%
\omega_{\varepsilon}(r,\tilde{X}_{r})drds+\int_{t}^{t+\Delta}\int_{t}^{s}%
L_{a}^{1}L_{a}^{0}\omega_{\varepsilon}(r,\tilde{X}_{r})d\langle B\rangle
_{r}ds\\
&  \text{ \  \ }+\int_{t}^{t+\Delta}\int_{t}^{s}L_{a}^{0}L_{a}^{1}%
\omega_{\varepsilon}(r,\tilde{X}_{r})drd\langle B\rangle_{s}+\int
_{t}^{t+\Delta}\int_{t}^{s}L_{a}^{1}L_{a}^{1}\omega_{\varepsilon}(r,\tilde
{X}_{r})d\langle B\rangle_{r}d\langle B\rangle_{s},
\end{split}
\label{A.2}%
\end{equation}
where%
\begin{align*}
L_{a}^{0}u(t,x)  &  =\partial_{t}u(t,x)+\sum_{i=1}^{m}b(t,x,a)\partial_{x_{i}%
}u(t,x)\text{,}\\
L_{a}^{1}u(t,x)  &  =\sum_{i=1}^{m}h_{i}(t,x,a)\partial_{x_{i}}u(t,x)+\frac
{1}{2}\sum_{i,j=1}^{m}[\sigma \sigma^{\top}]_{i,j}\left(  t,x,a\right)
\partial_{x_{i}x_{j}}^{2}u(t,x)\text{.}%
\end{align*}
Noting that $\mathbb{\hat{E}}[L_{a}^{1}\omega_{\varepsilon}(t,x)\Delta \langle
B\rangle]=2G(L_{a}^{1}\omega_{\varepsilon}(t,x))\Delta$, by inserting
(\ref{A.2}) into (\ref{A.1}), we deduce that for any $a\in A$, $\left(
t,x\right)  \in \lbrack0,T-\Delta]\times \mathbb{R}^{m},$%
\begin{equation}
L_{a}^{0}\omega_{\varepsilon}(t,x)+2G(L_{a}^{1}\omega_{\varepsilon
}(t,x))+f(t,x,a)\leq C\varepsilon^{-3}\Delta, \label{A.3}%
\end{equation}
where we have used the bound of the term $L_{a}^{i}f,L_{a}^{i}L_{a}^{j}%
\omega_{\varepsilon}$, $i,j=1,2$.

Step 2. Upper bound on $\tilde{v}(t,x)-v_{\Delta}(t,x)$. For any $\left(
t,x\right)  \in \lbrack0,T-\Delta)\times \mathbb{R}^{m},\alpha \in \mathcal{A}$,
by $G$-It\^{o}'s formula, we have%
\[
\mathbb{\hat{E}}_{t,x}^{\alpha}\big[\omega_{\varepsilon}(T-\Delta,\tilde
{X}_{T-\Delta})\big]=\omega_{\varepsilon}(t,x)+\mathbb{\hat{E}}_{t,x}^{\alpha
}\bigg[\int_{t}^{T-\Delta}L_{\alpha}^{0}\omega_{\varepsilon}(s,\tilde{X}%
_{s})ds+\int_{t}^{T-\Delta}L_{\alpha}^{1}\omega_{\varepsilon}(s,\tilde{X}%
_{s})d\langle B\rangle_{s}\bigg].
\]
In view of Corollary 3.5.8 in Peng \cite{P2010}, we obtain that, for each
$\eta \in M_{G}^{1}(0,T)$,
\begin{equation}
\int_{t}^{T-\Delta}\eta_{s}d\langle B\rangle_{s}\leq2\int_{t}^{T-\Delta}%
G(\eta_{s})ds. \label{A.4}%
\end{equation}
Combining with\ (\ref{A.0.2}), (\ref{A.3}), and (\ref{A.4}), from the
definition of $G(\cdot)$, we can deduce%
\begin{align*}
\mathbb{\hat{E}}_{t,x}^{\alpha}\big[\omega_{\varepsilon}(T-\Delta,\tilde
{X}_{T-\Delta})\big]  &  \leq \omega_{\varepsilon}(t,x)+\mathbb{\hat{E}}%
_{t,x}^{\alpha}\bigg[\int_{t}^{T-\Delta}\{L_{\alpha}^{0}\omega_{\varepsilon
}(s,\tilde{X}_{s})+2G(L_{\alpha}^{1}\omega_{\varepsilon}(s,\tilde{X}%
_{s}))\}ds\bigg]\\
&  \leq v_{\Delta}(t,x)-\mathbb{\hat{E}}_{t,x}^{\alpha}\bigg[\int_{t}%
^{T}f(s,\tilde{X}_{s},\alpha_{s})ds\bigg]+C(\varepsilon+\varepsilon^{-3}%
\Delta).
\end{align*}
Then, by (\ref{A.0.1}), we have
\begin{align*}
\mathbb{\hat{E}}_{t,x}^{\alpha}\big[g(X_{T})\big]  &  \leq \mathbb{\hat{E}%
}_{t,x}^{\alpha}\big[\omega_{\varepsilon}(T-\Delta,\tilde{X}_{T-\Delta
})\big]+\mathbb{\hat{E}}_{t,x}^{\alpha}\big[\omega_{\varepsilon}%
(T,X_{T})-\omega_{\varepsilon}(T-\Delta,\tilde{X}_{T-\Delta})\big]\\
&  \leq v_{\Delta}(t,x)-\mathbb{\hat{E}}_{t,x}^{\alpha}\bigg[\int_{t}%
^{T}f(s,\tilde{X}_{s},\alpha_{s})ds\bigg]+C(\varepsilon+\Delta^{\frac{1}{2}%
}+\varepsilon^{-3}\Delta),
\end{align*}
which implies%
\[
\tilde{v}(t,x)\leq v_{\Delta}(t,x)+C(\varepsilon+\Delta^{\frac{1}{2}%
}+\varepsilon^{-3}\Delta).
\]

For any $\left(  t,x\right)  \in \lbrack T-\Delta,T]\times \mathbb{R}^{m}%
,\alpha \in \mathcal{A}[t,T]$,
\[
\tilde{v}(t,x)-g(x)\leq \mathbb{\hat{E}}_{t,x}^{\alpha}\left[  |X_{T}%
-x|\right]  +C\Delta \leq C\Delta^{\frac{1}{2}},
\]
and similarly, we have $v_{\Delta}(t,x)-g(x)\leq C\Delta^{\frac{1}{2}}$. Thus
\[
\tilde{v}(t,x)-v_{\Delta}(t,x)\leq \tilde{v}(t,x)-g(x)-(v_{\Delta
}(t,x)-g(x))\leq C(\varepsilon+\Delta^{\frac{1}{2}}).
\]
To sum up, for any $\left(  t,x\right)  \in \lbrack0,T]\times \mathbb{R}^{m}$,
we get
\begin{equation}
\tilde{v}(t,x)-v_{\Delta}(t,x)\leq C(\varepsilon+\Delta^{\frac{1}{2}%
}+\varepsilon^{-3}\Delta). \label{A.5}%
\end{equation}

Step 3. Together with (\ref{A.0}) and (\ref{A.5}), we have for any $\left(
t,x\right)  \in \lbrack0,T]\times \mathbb{R}^{m}$,%
\[
v(t,x)-v_{\Delta}(t,x)\leq C(\varepsilon+\Delta^{\frac{1}{2}}+\varepsilon
^{-3}\Delta).
\]
By taking $\varepsilon=\Delta^{1/4}$, the desired result follows.
\end{proof}

\end{document}